\documentclass[11pt]{amsart}
\usepackage[utf8]{inputenc}

\usepackage{amssymb}
\usepackage{amsmath}
\usepackage{amsthm}
\usepackage{hyperref}
\usepackage[margin=1in]{geometry}
\usepackage{ytableau}
\usepackage{tikz}
\usepackage{mathtools}
\usepackage{derivative}
\usepackage{comment}

\newtheorem{theorem}{Theorem}[section]
\newtheorem{lemma}[theorem]{Lemma}

\newtheorem{definition}[theorem]{Definition}
\newtheorem{remark}[theorem]{Remark}

\newcommand{\ZZ}{\mathbb Z}

\newcommand{\CC}{\mathbb C}

\newcommand{\g}{\widehat{\mathfrak{sl}_2}}
\newcommand{\lam}{\lambda}
\newcommand{\Chess}{\text{Chess}}
\newcommand{\Tr}{\text{Tr}}
\newcommand{\Fc}{\mathcal F}
\DeclarePairedDelimiter\ket{\lvert}{\rangle}
\newcommand{\eps}{\varepsilon}
\title{Chess tableaux, powers of two and affine Lie algebras}
\author{Antoine Labelle}
\address{Department of Mathematics and Statistics \\ McGill University}
\email{antoine.labelle@mail.mcgill.ca}

\author{Stoyan Dimitrov}
\address{Department of Mathematics \\ Rutgers University}
\email{emailtostoyan@gmail.com}

\begin{document}

\begin{abstract}
    Chess tableaux are a special kind of standard Young tableaux where, in the chessboard coloring of the Young diagram, even numbers always appear in white cells and odd numbers in black cells. If, for $\lambda$ a partition of $n$, $\Chess(\lam)$ denotes the number of chess tableaux of shape $\lambda$, then Chow, Eriksson and Fan observed that $\displaystyle\sum_{\lam \vdash n} \Chess(\lam)^2$ is divisible by unusually large powers of $2$. In this paper, we give an explanation for this phenomenon, proving a lower bound of $n-O(\sqrt{n})$ for the $2$-adic valuation of this sum and a generalization of it. We do this by exploiting a connection with a certain representation of the affine Lie algebra $\g$ on the vector space with basis indexed by partitions. Our result about chess tableaux then follows from a study of the basic representation of $\g$ with coefficients taken from the ring of rational numbers with odd denominators.
\end{abstract}

\maketitle

%\tableofcontents

\section{Introduction}

A \emph{partition} is a nonincreasing finite sequence of positive integers $\lam = (\lam_1, \ldots, \lam_{\ell})$. The $\lam_i$ are called the parts of $\lam$. Given a partition $\lam$, the size of $\lam$, denoted $|\lam|$ is the sum of the parts of $\lam$. We also write $\lam \vdash n$ to mean "$\lam$ has size $n$". We also define the length of partition $\lam$, denoted $\ell(\lam)$, as the number of parts of $\lam$. Given $k \in \ZZ_{\ge 1}$, we denote by $m_k(\lam)$ the multiplicity of $k$ as a part of $\lam$ (possibly $0$ if $k$ does not appear in $\lam$).

We identify a partition with its \emph{Young diagram}, which consists of a series of left-aligned rows of boxes, with the $i^\text{th}$ row from the top having $\lam_i$ boxes (for example, Figure \ref{fig:chess-tab} shows the Young diagram for the partition $(6,4,1)$). 

Given a partition $\lam$ of size $n$, a \emph{standard Young tableau} is a filling of the cells of $\lam$ with the integers from $1$ to $n$ such that every integer between $1$ and $n$ appears exactly once and such that the entries in each row and column are increasing from left to right and from top to bottom.

A \emph{chess tableau} is a standard Young tableau in which the entry in cell $(i,j)$ has the parity of $i + j + 1$. Figure \ref{fig:chess-tab} shows an example of a chess tableau, where the cells with odd entries are shown in grey and the cells with even entries in white. The cells are colored as in a chessboard, which explains the origin of the name.

\begin{figure}[ht!]
    \centering
    \ytableausetup{centertableaux, boxsize=30pt} \begin{ytableau} *(gray) 1 & 2 & *(gray) 3 & 6 & *(gray) 7 & 10 \\ 4 & *(gray) 5 & 8 & *(gray) 9 \\*(gray) 11 \end{ytableau}
    \caption{A chess tableau with $11$ cells}
    \label{fig:chess-tab}
\end{figure}

In \cite{CEF}, Chow, Eriksson and Fan study chess tableaux and specifically the enumeration of chess tableaux of a fixed shape $\lam$. Their main result is a bijection between Chess tableaux on Young diagrams with three rows and certain so-called "nonconsecutive tableaux". For a given partition $\lam$, the total number of chess tableaux of shape $\lam$ is denoted by $\Chess(\lam)$.
In the conclusion, the authors mention the intriguing, but unexplained, observation that the quantity $\displaystyle\sum_{\lam \vdash n} \Chess(\lam)^2$ is divisible by unusually high powers of $2$, where the sum is over all partitions of size $n$. For example, the first few values of $n$ give the following values for $\displaystyle\sum_{\lam \vdash n} \Chess(\lam)^2$ :

$$1,~
2,~ 2,~ 2^2 ,~ 2^3 ,~ 2^4 ,~ 2^4 \cdot 3,~ 2^5 \cdot 5,~ 2^6 \cdot 7,~ 2^{11} ,~ 2^8 \cdot 5^2 ,~
2^9 \cdot 61,~ 2^{10} \cdot 3 \cdot 41,$$
$$ 2^{11} \cdot 5 \cdot 59,~ 2^{11} \cdot 1523,~ 2^{13} \cdot 23 \cdot 83,~
2^{13} \cdot 11411,~ 2^{15} \cdot 103 \cdot 163,~ \ldots$$

In this paper, we give an explanation for this phenomenon by exploiting a connection of chess tableaux with the representation theory of the affine Lie algebra $\g$. In fact, we prove a more general statement concerning variants of chess tableaux. To formulate it, we will need the following definition.

\begin{definition}
Let $e$ be a positive integer.
Given a cell $(i,j)$ in the Young diagram of a partition, the \emph{$e$-residue} of the cell is $i-j$, seen as an element of $\ZZ/e\ZZ$, the integers modulo $e$.
Given a standard Young tableau $T$ of size $n$, the \emph{$e$-residue sequence} $(i_1, i_2, \ldots, i_n)$ of $T$ is defined by letting $i_k$ be residue of the cell containing $k$.

Given a sequence $v = (i_1, i_2, \ldots, i_n) \in (\ZZ/e\ZZ)^n$ and a partition $\lam \vdash n$, we let $C_e(v, \lam)$ be the number of  standard Young tableaux of shape $\lam$ with $e$-residue sequence $v$.
\end{definition}

For example, when $e=2$, the cells with residue $0$ in Figure \ref{fig:chess-tab} are the grey cells and the cells with residue $1$ are the white ones. Note that, in the case $e=2$ and $v=(0,1,0,1,\ldots)$, we recover exactly chess tableaux and $C_2(v, \lam)=\Chess(\lam)$. Our main result is the following :

\begin{theorem} \label{thm:v2sumsquares}
Let $a(n)$ be the number of triangular numbers between $1$ and $n$ (equivalently, $a(n)$ is the largest integer $m$ such that $\frac{m(m+1)}{2}\le n$). Given a positive integer $n$ and $v,w \in (\ZZ/2\ZZ)^n$, the sum
$$\sum_{\lam \vdash n} C_2(v, \lam) C_2(w, \lam)$$
is divisible by $2^{n-a(n)}$.
\end{theorem}

In the case $v=w=(0,1,0,1,0,\ldots)$, Theorem \ref{thm:v2sumsquares} implies that $\displaystyle\sum_{\lam \vdash n} \Chess(\lam)^2$ is divisible by $2^{n-a(n)}$, and thus it explains the observation of Chow, Eriksson and Fan that this sum is divisible by high powers of $2$. 

Chess tableaux, and more generally the notion residue sequences appear naturally in various fields like combinatorics, representation theory and mathematical physics. They arise for example in the modular representation theory of the symmetric group, and in the study of cyclotomic Hecke algebras (which are closely related) \cite{Mat}. Note, in particular, that a sum very similar to the sum $\sum_{\lam \vdash n} C_e(v,\lambda)C_e(w,\lambda)$ that we study appears in a graded dimension formula for a cyclotomic quiver algebra in \cite[Theorem 4.20]{BK}. For a connexion with the world of combinatorics, note that some relationships with Baxter permutations have been observed in \cite{CEF}.

Our proof of Theorem \ref{thm:v2sumsquares} relies crucially on a study of the representation theory of the affine Lie algebra $\g$. Note that a connection between chess tableaux and the loop group $\text{SL}_2(\CC[t,t^{-1}])$ has also been found previously by Scott \cite{Sco}, so it is not too surprising that the corresponding Lie algebra $\g$ also appears in the study of chess tableaux.

\begin{remark} \label{rmk:SYT}
If we denote by $\text{SYT}(\lam)$ the total number of standard Young tableaux of shape $\lambda$, then the sum $\sum\limits_{\lambda \vdash n} \text{SYT}(\lambda)^{2}$ is given by $n!$ (because of the Robinson-Schensted-Knuth correspondence \cite{Knu}) and a well-known fact is that the largest power of $2$ dividing $n!$ is $n-b(n)$, where $b(n)$ is the number of $1$s in the binary representation of $n$. Our result can be seen as an analog of this fact for chess tableaux. We should note that $b(n)\leq \log_{2} (n+1)$, while the number of triangular numbers $a(n)$ grows as $\sqrt{n}$. Therefore, the largest power of $2$ dividing the quantity $\sum\limits_{\lambda \vdash n} \Chess(\lambda)^{2}$ is actually smaller than the analogous quantity for standard Young tableaux.
\end{remark}

In section \ref{sec:defns}, we introduce the main algebraic objects that will be needed. In particular, we define the affine Lie algebra $\g$ and explain how it acts on the vector space $\Fc$ with basis indexed by partitions. We also introduce the basic representation of $\g$, which is an irreducible subrepresentation of $\Fc$. In Section \ref{sec:basic-rep-Z2}, we study the structure of the basic representation of $\g$ when using $\ZZ_{(2)}$ coefficients, where $\ZZ_{(2)}=\{\frac{a}{b}: a,b \in \ZZ, 2\nmid b\}$ is the localization of $\ZZ$ at the prime ideal $(2)$. In particular, we give a description of the $\ZZ_{(2)}$-module generated by the action of the generators of $\g$ on the "highest weight vector" (the basis vector corresponding to the empty partition). We show that the inner product of two elements of this module is highly divisible by $2$. Theorem \ref{thm:v2sumsquares} then follows easily from these algebraic results.

\section{Definitions of the main objects} \label{sec:defns}

In the following subsections, we introduce the necessary background for our results. We follow mainly the sources \cite{Kac} and \cite{Tin}.

\subsection{The affine Lie algebra $\g$}

In this section, we introduce the Lie algebra $\g$. This is the simplest case of a so-called "affine Lie algebra"\footnote{More precisely, in the terminology of \cite{Kac}, this is the affine Lie algebra of type $A_1^{(1)}$.}, a special class of infinite dimensional Lie algebras with a rich theory, that is known in particular to have connections with the combinatorics of Young tableaux \cite{Ari}. For our purposes, the following explicit description from \cite[Theorem 8.7]{Kac} will be the most convenient.

\begin{definition}
    Let $\mathfrak{sl}_2$ be the Lie algebra of $2 \times 2$ matrices over $\CC$ with zero trace, with the usual bracket $[X,Y]=XY-YX$. Define a $\ZZ/2\ZZ$-grading on $\mathfrak{sl}_2$ by setting
    \begin{align*}
        (\mathfrak{sl}_2)_0 &= \CC \begin{pmatrix} 1 & 0 \\ 0 & -1 \end{pmatrix} \\
        (\mathfrak{sl}_2)_1 &= \CC \begin{pmatrix} 0 & 1 \\ 0 & 0 \end{pmatrix} \oplus \CC \begin{pmatrix} 0 & 0 \\ 1 & 0 \end{pmatrix}.
    \end{align*}
    It is immediate to check that this indeed defines a grading for $\mathfrak{sl}_2$.
\end{definition}

\begin{definition}
    Let 
    \[\g' = \left(\bigoplus_{k \in \ZZ} t^k (\mathfrak{sl}_2)_{\bar{k}} \right)\oplus \CC c\]
    where $\bar{k}$ is the image of $k$ in $\ZZ/2\ZZ$, and $t^n$ and $c$ are formal symbols. This is a Lie algebra with bracket defined by
    \begin{align*}
        [t^m X, t^n Y] &= t^{m+n}[X,Y] + \frac{m}{2} \Tr(XY) \delta_{m+n}c \\
        [c, Z] &= 0
    \end{align*}
    for all $X, Y \in \mathfrak{sl}_2$ and $Z \in \g'$, where $\delta_{m+n}$ is the Kronecker delta (i.e. $\delta_n$ is $1$ if $n=0$ and $0$ otherwise). Let $\g = \g' \oplus \CC d_0$, where $d_0$ acts by 
    \[ [d_0, t^nX] = nt^nX, \]
    \[ [d_0, c] = 0. \]
    
    This Lie algebra comes with a natural $\ZZ$-grading, where the degree $k$ part is $t^k (\mathfrak{sl}_2)_{\bar{k}}$ for $k\ne 0$ and $(\mathfrak{sl}_2)_{0} \oplus \CC c \oplus \CC d_0$ for $k=0$.
\end{definition}

The derived algebra $\g'$ has generators $e_0, e_1, f_0, f_1$, called the Chevalley generators, given by
\begin{alignat*}{2}
    e_0 &= t\begin{pmatrix} 0 & 0 \\ 1 & 0\end{pmatrix} &\qquad e_1 &= t\begin{pmatrix} 0 & 1 \\ 0 & 0\end{pmatrix} \\ 
    f_0 &= t^{-1}\begin{pmatrix} 0 & 1 \\ 0 & 0\end{pmatrix} &\qquad f_1 &= t^{-1}\begin{pmatrix} 0 & 0 \\ 1 & 0\end{pmatrix}.
\end{alignat*}

\begin{definition} \label{def:involution}
    The \emph{Cartan involution} $\omega$ is the involutive automorphism of $\g$ defined by
    \begin{align*}
        \omega(t^nX)&=-t^{-n}X^T \\
        \omega(c)&=-c \\
        \omega(d_0)&=-d_0
    \end{align*}
    where $X\in (\mathfrak{sl}_2)_{\bar{n}}$ and $X^T$ is the transpose of $X$.
\end{definition}

\subsection{The representation of $\g$ on the Fock space}

Let $\Fc$ be the complex vector space with basis elements $\ket{\lam}$ indexed by partitions. Following the terminology in \cite{Tin} or \cite{LLT}, we call $\Fc$ the Fock space (a name originating from mathematical physics). We endow $\Fc$ with the inner product for which the $\ket{\lam}$ form an orthonormal basis, and with the grading for which the degree of $\ket{\lambda}$ is the size of $\lambda$. The Lie algebra $\g$ is related to the combinatorics of residue sequences through the following representation on $\Fc$ \cite{Lec, LLT, Tin}:

\begin{theorem}
    There exist a representation of $\g$ on $\Fc$ where the action of the generators is given as follows. For $i=0,1$,
    
    $$f_i \ket{\lam} = \sum_{\lam^+} \ket{\lam^+},$$
    where the sum is over all partitions $\lam^+$ obtained from $\lam$ by adding a cell of $2$-residue $i$,
    $$e_i \ket{\lam} = \sum_{\lam^-} \ket{\lam^-},$$
     where the sum is over all partitions $\lam^-$ obtained from $\lam$ by removing a cell of $2$-residue $i$ and
    $$d_0 \ket{\lam} = n \ket{\lam},$$
    where $\lam$ has size $n$.
\end{theorem}

This representation is not irreducible, but one can show that the subrepresentation $R = U(\g)\ket{\emptyset}$ of $\Fc$ generated by the empty partition is irreducible \cite[Section 3]{Lec}, and actually isomorphic to the basic representation of $\g$ \cite[Proposition 3.35]{Tin} (where $U(\g)$ is the universal enveloping algebra). The basic representation will be important for us, so we give a precise description of it in the next section. 

\subsection{The basic representation}

As mentioned in the previous section, the subrepresentation $R = U(\g)\ket{\emptyset}$ of $\Fc$ generated by the empty partition is isomorphic to the basic representation of $\g$. In this section, we recall the description of this representation in terms of vertex operators from \cite[Chapter 14.7]{Kac}\footnote{Our notation here differs a little bit from that of \cite{Kac}. There, the polynomial ring is called $\CC[x_1, x_3, x_5, \ldots]$, and our $p_k$ corresponds to $kx_k$, but it is easy to translate all their formulas to get ours. This is done to simplify certain expressions and to match the conventions for symmetric functions (see Remark \ref{rmk:symm-fun}). }

\begin{definition}
    For $j\in\ZZ$, define
    \begin{equation*}
    A_j = t^j \cdot
    \begin{cases}
        \begin{pmatrix} -1 & 0 \\ 0 & 1 \end{pmatrix} \quad &\text{if } j\equiv 0 \pmod{2} \vspace{2mm} \\
        
        \begin{pmatrix} 0 & -1 \\ 1 & 0 \end{pmatrix} \quad &\text{if } j\equiv 1 \pmod{2}
    \end{cases}
    \end{equation*}
    Also define, for odd $k \ge 1$,
    \begin{align*}
        \alpha_k &= t^k \begin{pmatrix} 0 & 1 \\ 1 & 0 \end{pmatrix} \\
        \alpha_k^* &= t^{-k} \begin{pmatrix} 0 & 1 \\ 1 & 0 \end{pmatrix}
    \end{align*}
\end{definition}

Note that the $A_j$ for $j\in \ZZ$, the $\alpha_k$ and $\alpha_k^*$ for odd $k\ge 1$ and $c$ form a basis of $\g'$.

\begin{theorem} \label{thm:basic-rep}
R can be identified with the infinite polynomial ring $\CC[p_1, p_3, p_5, \ldots]$, with $\ket{\emptyset}$ identified with $1$. Under this identification, $c$ acts as the identity, $\alpha_k^*$ acts as multiplication by $p_k$, and $\alpha_k$ acts as $k\pdv{}{p_k}$ (formal differentiation of polynomials). Moreover we have the equality of formal power series
\begin{equation} \label{eq:vertex-operator}
    \sum_{j\in \ZZ} z^{-j} A_j = \frac{1}{2}\exp\left(-2 \sum_{\substack{k \ge 1\\\text{odd}}} \frac{p_k}{k} z^k\right)\exp\left(2 \sum_{\substack{k \ge 1\\\text{odd}}} \pdv{}{p_k}z^{-k}\right).
\end{equation}
The grading on $\Fc$ makes $R$ into a graded ring, where $p_k$ has degree $k$.
\end{theorem}

In (\ref{eq:vertex-operator}), $p_k$ is to be interpreted as the operator of multiplication by $p_k$. More generally, given $f\in R$, we will often write simply $f$ to denote the "multiplication by $f$" operator. Note that, in the expansion of the right-hand side, the coefficient of any power of $z$ is an infinite sum of operators, but any $v\in R$ is annihilated by all but finitely many of the terms, so the infinite sum is a well-defined operator. 

Note also that our inner product on $\Fc$ induces an nondegenerate bilinear form on $R$, satisfying $(1,1)=1$ and 
\begin{equation}
    (Xu, v)=-(u, \omega(X)v)
\end{equation}
for all $X \in \g$, $u,v \in R$ (recall Definition \ref{def:involution} for the definition of the involution $\omega$). By \cite[Proposition 9.4]{Kac}, such a form is unique and an explicit description of this invariant form on $\CC[p_1, p_3, p_5, \ldots]$ can be given \cite[Equation 14.12.6]{Kac}. When $\mu=(\mu_1, \mu_2, \cdots, \mu_{\ell})$ is a partition with odd parts, we will denote by $p_\mu$ the monomial $p_{\mu_1}p_{\mu_2}\cdots p_{\mu_{\ell}}$.

\begin{theorem} \label{thm:R-inner-prod}
The monomials $p_\mu$ form an orthogonal basis of $R$ (i.e. $(p_\mu, p_\nu)=0$ for $\mu \ne \nu$), and
\begin{equation*}
    (p_\mu, p_\mu) = z_\mu
\end{equation*}
where
\begin{equation} \label{eq:z-def}
z_\mu = \prod_{\substack{k \ge 1\\\text{odd}}} k^{m_k}m_k!.
\end{equation}
for $m_k=m_k(\mu)$.\footnote{Note that this agree with the formula \cite[(4.7)]{Mac} for the Hall inner product on symmetric function in the power-sum basis. See Remark \ref{rmk:symm-fun} for the explanation of this.}
\end{theorem}

If we expand the first exponential in (\ref{eq:vertex-operator}), we get a term $\frac{(-2)^\ell p_{k_1} \cdots p_{k_\ell}}{\ell!\cdot  k_1 \cdots k_\ell}z^{k_1+ \cdots + k_\ell}$ for every $\ell\ge 0$ and tuple $(k_1,\ldots, k_\ell)$ where the $k_i$ are odd positive integers. Moreover, if $\mu =(k_1, \ldots, k_\ell)$ is a partition with odd parts, there are $\frac{\ell!}{m_1(\mu)!m_3(\mu)!m_5(\mu)!\ \cdots}$ ways to reorder the tuple $(k_1, \ldots, k_\ell)$ to get the same term proportional to $p_\mu z^{|\mu|}$, so that the overall coefficient of $p_\mu z^{|\mu|}$ in the expansion is $\frac{(-1)^\ell 2^\ell}{z_\mu}=\frac{(-1)^{|\mu|} 2^\ell}{z_\mu}$. In other words, we have 

\begin{equation} \label{eq:expand-exp}
    \exp\left(-2 \sum_{\substack{k \ge 1\\\text{odd}}} \frac{p_k}{k} z^k\right) = \sum_{n\ge 0} q_n (-z)^n
\end{equation}

where

\begin{equation} \label{eq:Qn-sum}
    q_n=\sum_{\substack{\mu \vdash n\\\text{with odd}\\ \text{parts}}} \frac{2^{\ell(\mu)}}{z_\mu}p_\mu
\end{equation}

Given $f \in \CC[p_1, p_3, p_5, \ldots]$, we let $f^*$ denote the adjoint of the multiplication by $f$ operator (so that $(fv,w)=(v,f^*w)$ for all $v,w \in \CC[p_1, p_3, p_5, \ldots]$). From Theorem \ref{thm:R-inner-prod}, we can see that $p_k^*= k\pdv{}{p_k}$, so we can rewrite (\ref{eq:vertex-operator}) as

\begin{align}
    \sum_{j\in \ZZ} z^{-j} A_j &= \frac{1}{2}\exp\left(-2 \sum_{\substack{k \ge 1\\\text{odd}}} \frac{p_k}{k} z^k\right)\exp\left(2 \sum_{\substack{k \ge 1\\\text{odd}}} \pdv{}{p_k}z^{-k}\right) \nonumber \\
    &= \frac{1}{2}\exp\left(-2 \sum_{\substack{k \ge 1\\\text{odd}}} \frac{p_k}{k} z^k\right)\exp\left(-2 \sum_{\substack{k \ge 1\\\text{odd}}}  \frac{p_k^*}{k}(-z^{-1})^k\right) \nonumber \\
    &= \frac{1}{2} \left(\sum_{n\ge 0} q_n(-z)^n\right)\left(\sum_{n\ge 0} q_n^*z^{-n}\right) \label{eq:sum-Qn-sum-Qn*}
\end{align}
        
where the last equality follows from (\ref{eq:expand-exp}) together with the equation obtained by taking the adjoint on both sides of (\ref{eq:expand-exp}) and substituting $z$ for $-z^{-1}$. We will make crucial use of equation (\ref{eq:sum-Qn-sum-Qn*}) later to prove Theorem \ref{thm:Delta-stable}.

\begin{remark} \label{rmk:symm-fun}
We defined the vector space $\Fc$ as an abstract vector space with basis indexed by partitions. There is actually a very important concrete vector space in algebraic combinatorics which also has a basis indexed by partitions, namely the algebra of symmetric functions in infinitely many variables $x_1,x_2, \ldots$ \cite{Mac}  \cite[Chapter 2]{GR}. We can identify $\Fc$ with this algebra by identifying the basis $\{\ket{\lam}\}$ with the basis of Schur functions $\{s_\lam\}$. Under this identification, the subspace $R$ gets a very nice interpretation: the $p_k$'s correspond to the power-sum symmetric functions $x_1^k+x_2^k+\cdots$ \cite{Lec}. The $q_n$'s also correspond to known symmetric functions, namely the Hall-Littlewood $q$-functions at $t=-1$ (compare (\ref{eq:Qn-sum}) with example 6a from \cite[Chapter 3.8]{Mac}). We will, however, not need these interpretations in the rest of the paper.
\end{remark}

\section{Study of the basic representation over $\ZZ_{(2)}$} \label{sec:basic-rep-Z2}

Recall that $\ZZ_{(2)}=\{\frac{a}{b}: a,b \in \ZZ, 2\nmid b\}$ is the ring of rational numbers with odd denominator (i.e. the localization of $\ZZ$ at the prime ideal $(2)$). In this section, we study the $\ZZ_{(2)}$-module generated by the action of the generators of $\g$ on the constant polynomial $1\in \CC[p_1, p_3, p_5, \ldots]$. That is, we give an explicit description of $\ZZ_{(2)}[e_0,e_1,f_0,f_1]\cdot 1$, where $\ZZ_{(2)}[e_0,e_1,f_0,f_1]\subset U(\g)$ is the sub-$\ZZ_{(2)}$-algebra of $U(\g)$ generated by $e_0,e_1,f_0, f_1$.

\begin{definition} 
    We define $\Delta = \mathbb{Z}_{(2)}[2^k  p_{2k+1} : k \ge 0] = \mathbb{Z}_{(2)}[p_1, 2p_3, 4p_5, \ldots]$. As a $\ZZ_{(2)}$-module, this has a basis given by $2^\frac{|\mu|-\ell(\mu)}{2} p_\mu$, where $\mu$ runs over all partitions with odd parts.
\end{definition}

\subsection{Stability of $\Delta$}

Let us introduce the notation
\begin{equation*}
\varepsilon(n) = 
\begin{cases}
    0 \quad &\text{if $n$ is even} \\
    1 \quad &\text{if $n$ is odd}
\end{cases},
\end{equation*}
which will be convenient in what follows.

In order to prove the main result of this section, Theorem \ref{thm:Delta-stable}, we will need the following two lemmas.

\begin{lemma} \label{lem:Qn-Delta}
For $n\ge 1$, the "multiplication by $q_n$" operator sends $\Delta$ to $2^{-\frac{n+\eps(n)-4}{2}}\Delta$
\end{lemma}

\begin{proof}
We check that each term $\frac{2^{\ell(\mu)}}{z_\mu}p_\mu$ in the decomposition (\ref{eq:Qn-sum}) for $q_n$ sends $\Delta$ to $2^{-\frac{n+\eps(n)-4}{2}}\Delta$, where $\mu$ is a partition of $n$ with odd parts.
Since $\Delta$ is closed under multiplication, it suffices to verify that $\frac{2^{\ell(\mu)}}{z_\mu}p_\mu \in 2^{-\frac{n+\eps(n)-4}{2}}\Delta$, i.e. that
\begin{equation} \label{eq:v2-inequality}
\ell(\mu) - v_2(z_\mu) \ge \frac{n-\ell(\mu)}{2} - \frac{n+\eps(n)-4}{2} = \frac{4-\ell(\mu)-\eps(n)}{2}.
\end{equation}

where $v_2(x)$ denotes the $2$-adic valuation of $x$, i.e. the exponent of $2$ in the prime factorization of $x$.

Using the inequality $v_2(m!)\le m-1$ for $m\ge 1$ (which is well-known and follows from Legendre's formula \cite[Theorem 2.6.4]{Mol}), we have that

\[v_2(z_\mu) = \sum_{k\in S} v_2(m_k(\mu)!) \le \sum_{k\in S} (m_k(\mu)-1) = \ell(\mu)-|S| \le \ell(\mu)-1 \]
where $S$ is the set of parts of $\mu$ (ignoring multiplicities). Hence the left-hand side of (\ref{eq:v2-inequality}) is at least $1$.

Moreover, we clearly have $\ell(\mu)\ge 2-\eps(n)$, from which we get that the right-hand side of (\ref{eq:v2-inequality}) is at most $1$, which concludes the proof.

\end{proof}

\begin{lemma} \label{lem:Qn-star-Delta}
For $n\ge 1$, the operator $q_n^*$ sends $\Delta$ to $2^ \frac{n-\eps(n)+2}{2}\Delta$
\end{lemma}

\begin{proof}
By linearity, it suffices to check that \[\left(\frac{2^{\ell(\mu)}}{z_\mu}p_\mu^*\right) \cdot \left(2^\frac{|\tau|-\ell(\tau)}{2} p_\tau \right) \in 2^ \frac{n-\eps(n)+2}{2}\Delta \]
 for any partitions $\mu, \tau$ with odd parts where $\mu$ has size $n$. Using $p_k^*=k\pdv{}{p_k}$, we see that the left-hand side is zero unless $\tau=\mu \sqcup \nu$ for some $\nu$, where $\mu \sqcup \nu$ is the partition with parts $\mu_1, \mu_2, \ldots, \nu_1, \nu_2, \ldots$ (i.e. $p_\tau = p_\mu p_\nu$). In that case, we have $p_\mu^*(p_\mu p_\nu)=Cp_\nu$, where $C = \displaystyle\prod_k \frac{(m_k(\mu)+m_k(\nu))!}{m_k(\nu)!}$.
 
Note that 
\[v_2\left(\frac{C}{z_\mu}\right) = \sum_k v_2\left(\frac{(m_k(\mu)+m_k(\nu))!}{m_k(\nu)!m_k(\mu)!}\right)\ge 0\]
since $\frac{(m_k(\mu)+m_k(\nu))!}{m_k(\nu)!m_k(\mu)!}={m_k(\mu)+m_k(\nu) \choose m_k(\nu)}$ is an integer. Hence we have
\[ \left(\frac{2^{\ell(\mu)}}{z_\mu}p_\mu^*\right) \cdot \left(2^\frac{|\tau|-\ell(\tau)}{2} p_\tau \right) = 2^{\ell(\mu)}\cdot 2^\frac{|\mu|-\ell(\mu)}{2} \cdot 2^\frac{|\nu|-\ell(\nu)}{2} \cdot \frac{p_\mu^*(p_\mu p_\nu)}{z_\mu} = 2^\frac{n+\ell(\mu)}{2} \cdot 2^\frac{|\nu|-\ell(\nu)}{2} \cdot \frac{Cp_\nu}{z_\mu} \in 2^\frac{n+\ell(\mu)}{2}\Delta\]
as $\frac{C}{z_\mu}\in \ZZ_{(2)}$ and $2^\frac{|\nu|-\ell(\nu)}{2} p_\nu\in \Delta$.
The lemma then follows from $\ell(\mu) \ge 2-\eps(n)$.
\end{proof}

\begin{theorem} \label{thm:Delta-stable}
The module $\Delta$ is stable under the action of the generators $e_0, e_1, f_0, f_1$, i.e. $e_0 v \in \Delta$ for all $v\in \Delta$, and similarly for $e_1,f_0,f_1$.
\end{theorem}

\begin{proof}
We first treat the case of $f_0$ and $f_1$.
From Theorem \ref{thm:basic-rep}, $f_0+f_1=\alpha_1^*$ acts on $R$ by multiplication by $p_1$, while $f_1-f_0=A_{-1}$ acts as the coefficient of $z$ in (\ref{eq:sum-Qn-sum-Qn*}), which 
is $\frac{1}{2} \displaystyle\sum_{n=0}^\infty (-1)^{n+1} q_{n+1}q_{n}^*$. Hence we have
\begin{align*}
    f_0 = \frac{\alpha_1^* - A_{-1}}{2} &= p_1 - \frac{1}{4}\displaystyle\sum_{n=1}^\infty (-1)^{n+1} q_{n+1}q_{n}^* \\
    f_1 = \frac{\alpha_1^* + A_{-1}}{2} &= \frac{1}{4}\displaystyle\sum_{n=1}^\infty (-1)^{n+1} q_{n+1}q_{n}^*
\end{align*}

Combining Lemmas \ref{lem:Qn-Delta} and \ref{lem:Qn-star-Delta}, we see that, for $n \ge 1$, the operator $q_{n+1}q_n^*$ sends $\Delta$ to $$2^{-\frac{n+1+\eps(n+1)-4}{2}}2^\frac{n-\eps(n)+2}{2}\Delta = 4 \Delta$$
which shows that $f_0$ and $f_1$ preserve $\Delta$.

Similarly, we have

\begin{align*}
    e_0 = \frac{\alpha_1 + A_{1}}{2} &= p_1^* - \frac{1}{4}\displaystyle\sum_{n=1}^\infty (-1)^{n+1} q_{n}q_{n+1}^* \\
    e_1 = \frac{\alpha_1 - A_{1}}{2} &= \frac{1}{4}\displaystyle\sum_{n=1}^\infty (-1)^{n+1} q_{n}q_{n+1}^*
\end{align*}

For $n \ge 1$, the operator $q_{n}q_{n+1}^*$ sends $\Delta$ to $$2^{-\frac{n+\eps(n)-4}{2}}2^\frac{n+1-\eps(n+1)+2}{2}\Delta = 8 \Delta$$

which shows that $e_0$ and $e_1$ preserve $\Delta$ (using that $p_1^*$ also clearly preserves $\Delta$).

\end{proof}

Theorem \ref{thm:Delta-stable} show that $\ZZ_{(2)}[e_0,e_1,f_0,f_1]\cdot 1 \subset \Delta$, where $1$ here represents the constant polynomial in $R=\CC[p_1,p_3, \ldots]$. The reverse inclusion also holds:

\begin{theorem} \label{thm:Delta-gen}
    $\ZZ_{(2)}[f_0,f_1]\cdot 1 =\ZZ_{(2)}[e_0,e_1,f_0,f_1]\cdot 1 = \Delta$
\end{theorem}

\begin{proof}
From Theorem \ref{thm:Delta-stable}, we have the inclusions
 \[\ZZ_{(2)}[f_0,f_1]\cdot 1 \subset \ZZ_{(2)}[e_0,e_1,f_0,f_1]\cdot 1 \subset \Delta,\]
 so it remains to show that
 \[ \Delta \subset \ZZ_{(2)}[f_0,f_1]. \]
One can easily check that, for $k\ge 1$ odd,
\[ \left[f_0, \left[f_1, t^{-k}\begin{pmatrix} 0 & 1 \\ 0 & 0 \end{pmatrix}\right]\right] = 2t^{-k-2}\begin{pmatrix} 0 & 1 \\ 0 & 0 \end{pmatrix},\]
\[ \left[f_1, \left[f_0, t^{-k}\begin{pmatrix} 0 & 0 \\ 1 & 0 \end{pmatrix}\right]\right] = 2t^{-k-2}\begin{pmatrix} 0 & 0 \\ 1 & 0 \end{pmatrix},\]
so by induction it follows that
\[ 2^mt^{-(2m+1)}\begin{pmatrix} 0 & 1 \\ 0 & 0 \end{pmatrix} = (\text{ad}_{f_0}\text{ad}_{f_1})^m f_0 \in \ZZ_{(2)}[f_0,f_1],\]
\[ 2^mt^{-(2m+1)}\begin{pmatrix} 0 & 0 \\ 1 & 0 \end{pmatrix} = (\text{ad}_{f_1}\text{ad}_{f_0})^m f_1 \in \ZZ_{(2)}[f_0,f_1].\]
Adding these two together, we find that
\[2^m\alpha_{2m+1}^*= 2^mt^{-(2m+1)}\begin{pmatrix} 0 & 1 \\ 1 & 0 \end{pmatrix}\in \ZZ_{(2)}[f_0,f_1].\]
We see therefore that $\ZZ_{(2)}[f_0,f_1]\cdot 1$ is closed under multiplication by $2^mp_{2m+1}$ for any $m\ge 0$, so it clearly contains $\Delta$.
\end{proof}

\subsection{The inner product on $\Delta$}

The last missing ingredient for our proof of Theorem \ref{thm:v2sumsquares} is to show that the inner product of two elements of $\Delta$ is highly divisible by $2$.

\begin{theorem} \label{thm:Delta-inner-prod}
For $f, g \in \Delta$ of degree $n$, we have $\langle f,g \rangle \in 2^{n-a(n)}\mathbb{Z}_{(2)}$
\end{theorem}

\begin{proof}
Note that $\Delta_n$ (the elements of degree $n$ in $\Delta$) has a $\mathbb{Z}_{(2)}$-basis consisting of all $2^\frac{n-\ell(\mu)}{2}p_\mu$ for $\mu$ a partition of $n$ with odd parts. Since the $p_\mu$ are pairwise orthogonal, it suffices to show that $$\langle 2^\frac{n-\ell(\mu)}{2}p_\mu, 2^\frac{n-\ell(\mu)}{2}p_\mu \rangle \in 2^{n-a(n)}\mathbb{Z}_{(2)}$$ for every $\mu$ with odd parts, that is,
$$n-\ell(\mu)+v_2(z_\mu)\ge n-a(n),$$
or equivalently
$$\ell(\mu)-v_2(z_\mu)\le a(n).$$

To do this, consider the following classical bijection between partitions of $n$ with odd parts and partitions of $n$ with distinct parts. Given a partition $\mu$ with odd parts and an odd positive integer $k$, let $m_k=m_k(\mu)$ be the multiplicity of $k$ in $\mu$ and let $m_k=2^{r_{k,1}}+\cdots+2^{r_{k,b(m_k)}}$ be the binary expansion of $m_k$ (where $b(m_k)$ is the number of terms in this expansion, or equivalently the number of ones when $m_k$ is written in binary). Let $\nu$ be the partition having parts $2^{r_{k,i}}k$, where $k$ runs over odd positive integers and $i$ runs from $1$ to $b(m_k)$. Then $\nu$ is a partition of $n$ with distinct parts, and it is not hard to check that this process can be reversed, so that it gives a bijection between partitions of $n$ with odd parts and with distinct parts.

Note that the number of parts of $\nu$ is equal to $\displaystyle\sum_k b(m_k)$. On the other hand, from the definition of $z_\mu$, we have that
$$v_2(z_\mu)=\sum_k v_2(m_k!)= \sum_k \left( m_k-b(m_k) \right)=\ell(\mu)-\sum_k b(m_k),$$
where we used Legendre's formula according to which $v_2(m!)=m-b(m)$ \cite[Theorem 2.6.4]{Mol}. Therefore, if $\nu$ corresponds to $\mu$ under the described bijection, then we have
$$\ell(\mu)-v_2(z_\mu)=\ell(\nu).$$
The problem is thus reduced to showing that $\ell(\nu)\le a(n)$ whenever $\nu$ is a partition of $n$ with distinct parts. But this is clear since, if $\nu$ has distinct parts, then $n = \nu_1 + \nu_2 + \cdots + \nu_{\ell(\nu)} \ge 1 + 2 + \cdots + \ell(\nu) = \frac{\ell(\nu)(\ell(\nu)+1)}{2}$. Since $a(n)$ is the largest integer $m$ such that $\frac{m(m+1)}{2}\le n$, it follows that $\ell(\nu)\le a(n)$.
\end{proof}

Our theorem on Young tableaux with fixed $2$-residue sequence is now an easy corollary of the results established so far: 

\begin{proof}[Proof of Theorem \ref{thm:v2sumsquares}]
Recall that, for $v \in (\ZZ/2\ZZ)^n$ and $\lam$ a partition of $n$, $C_2(v, \lam)$ denotes the number of standard Young tableaux of shape $\lam$ with $2$-residue sequence $v$.
We need to show that, given a positive integer $n$ and $v,w \in \{0,1\}^n$, the sum
$$\sum_{\lam \vdash n} C_2(v, \lam) C_2(w, \lam)$$
is divisible by $2^{n-a(n)}$.

Let $v=(i_1,i_2,\ldots,i_n)$ and $w=(j_1,j_2,\ldots,j_n)$. It is easy to see from the description of the action of $f_i$ on $\Fc$ by adding cells of residue $i$ that
$$\sum_{\lam \vdash n} C_2(v,\lam) \ket{\lam} = f_{i_n} \cdots f_{i_2}f_{i_1} \ket{\emptyset} \in \Delta$$
and similarly
$$\sum_{\lam \vdash n} C_2(w,\lam) \ket{\lam} = f_{j_n} \cdots f_{j_2}f_{j_1} \ket{\emptyset} \in \Delta,$$
using Theorem \ref{thm:Delta-stable}.
Therefore, taking the inner product, we find by Theorem \ref{thm:Delta-inner-prod} that
\begin{equation*}
\sum_{\lam \vdash n} C_2(v,\lam) C_2(w,\lam)=\left\langle \sum_{\lam \vdash n} C_2(v,\lam) \ket{\lam}, \sum_{\lam \vdash n} C_2(w,\lam) \ket{\lam} \right\rangle \in 2^{n-a(n)}\mathbb{Z}_{(2)}.
\end{equation*}
\end{proof}

\begin{remark}
It is easy to see from the proof that the bound in Theorem \ref{thm:Delta-inner-prod} is tight, i.e. the theorem is false if we replace $n-a(n)$ in the statement by any larger number. Since expressions of the form $f_{i_n} \cdots f_{i_2}f_{i_1} \ket{\emptyset}$ span the degree $n$ part of $\ZZ_{(2)}[f_0,f_1]\ket{\emptyset} = \Delta$, it follows that the bound of Theorem \ref{thm:v2sumsquares} has to be tight too, i.e. for every $n$ there must exist $v,w$ for which $\sum_{\lam \vdash n} C_2(v, \lam) C_2(w, \lam)$ is divisible exactly by $2^{n-a(n)}$.

For fixed $v,w$ (e.g. $v=w=(0,1,0,1, \ldots)$ in the case of chess tableaux), the sum might be divisible by a larger power of $2$, but computations suggest that it is generally not much larger.
\end{remark}

\section{Conclusion}

A natural question one can ask about Theorem \ref{thm:v2sumsquares} is whether it can be extended to any other $e\ne 2$. That is, given $v,w \in (\ZZ/e\ZZ)^n$, can we say anything about the divisibility properties of the sum $\displaystyle\sum_{\lam \vdash n} C_e(v,\lam)C_e(w,\lam)$?

If $e=1$, as mentioned in Remark \ref{rmk:SYT}, this comes down to counting all standard Young tableaux, and the sum is equal to $n!$, which is indeed highly divisible by any fixed prime $p$. Note that one way to prove that the sum is $n!$ is to use an action on $\Fc$ of the Heisenberg Lie algebra (the Lie algebra with basis $e,f,c$ with $[c,e]=[c,f]=0$ and $[f,e]=c$), which is analogous to our use of the action of $\g$ for $e=2$. On $\Fc$, $c$ acts as the identity, $f$ acts by the operator of adding a cell (of any color) to a partition and $e$ acts by removing a cell (see \cite[Chapter 8]{Sta} for details). 

If $e>2$, however, computer calculations of $\displaystyle\sum_{\lam \vdash n} C_e(v,\lam)C_e(w,\lam)$ did not reveal any divisibility pattern analogous to Theorem \ref{thm:v2sumsquares}. One explanation for why this phenomenon happens only for $e=2$ is that the basic representation over $\ZZ_{(2)}$ (the module $\Delta$) is "smaller than expected", i.e. it is not $\ZZ_{(2)}[p_1, p_3, p_5, ...]$ but rather $\ZZ_{(2)}[p_1, 2p_3, 4p_5, ...]$. This is sort of a coincidence that happens only for $e=2$; for larger $e$ we can still relate $e$-residue sequences to a representation of the affine algebra $\widehat{\mathfrak{sl}_e}$, but (at least when $e$ is prime) the $\ZZ_{(e)}$-module analogous to $\Delta$ appears to be simply $\ZZ_{(e)}[p_i : e \nmid i]$. The coincidence for $e=2$ can be traced back to the factor of $2$ in the commutation relations 

\[ \left[\begin{pmatrix} 1 & 0 \\ 0 & -1 \end{pmatrix}, \begin{pmatrix} 0 & 1 \\ 0 & 0 \end{pmatrix}\right] = 2\begin{pmatrix} 0 & 1 \\ 0 & 0 \end{pmatrix},\]
\[ \left[\begin{pmatrix} 1 & 0 \\ 0 & -1 \end{pmatrix}, \begin{pmatrix} 0 & 0 \\ 1 & 0 \end{pmatrix}\right] = -2\begin{pmatrix} 0 & 0 \\ 1 & 0 \end{pmatrix},\]
which are responsible for the emergence of factors of $2$ in the proof of Theorem \ref{thm:Delta-gen}. There is no analog of that for larger $e$.

Let us also mention that the form of the expression $\displaystyle\sum_{\lam \vdash n} \Chess(\lam)^2$ strongly suggests looking for a less algebraic and more combinatorial explanation for the powers of $2$ phenomenon into the RSK algorithm (which is indeed what the authors of \cite{CEF} suggest when they mention their observation). Under this viewpoint, the quantity $\displaystyle\sum_{\lam \vdash n} C_2(v,\lam)C_2(w,\lam)$ corresponds to the number of permutations $\pi$ of size $n$ which are mapped by RSK to a pair $(P,Q)$ of Young tableaux such that $P$ has $2$-residue sequence $v$ and $Q$ has $2$-residue sequence $w$. We have however not been able to find an alternative explanation of Theorem \ref{thm:v2sumsquares} that uses this viewpoint.

\section*{Acknowledgments}

We would like to thank Darij Grinberg for a useful answer on MathOverflow, which inspired our idea to try to describe exactly the $\ZZ_{(2)}$-module generated by the action of the generators of $\g$ (see \url{https://mathoverflow.net/q/403153/160416}), and for finding several typos. We would also like to thank Timothy Chow for letting us know about the status of the research on chess tableaux.

\bibliography{main.bib}
\bibliographystyle{plain}
\nocite{*}

\end{document}